\documentclass[a4paper]{article}
\usepackage{amsmath}
\usepackage{amssymb}
\usepackage{mathrsfs}
\usepackage{color}
\usepackage{amsmath}
\usepackage{amsfonts}
\usepackage{amsthm}
\usepackage{cases}
\usepackage{empheq}
\newcommand{\si}{\sigma}
\newcommand{\la}{\lambda}
\newcommand{\ol}{\overline}

\newcommand{\pa}{\partial}
\newcommand{\La}{\Lambda}
\newcommand{\al}{\alpha}
\newcommand{\be}{\beta}
\newcommand{\de}{\delta}

\newcommand{\ga}{\gamma}
\newcommand{\ka}{\kappa}

\newcommand{\cd}{\cdot}

\newcommand{\diag}{\operatornamewithlimits{diag}}
\newcommand{\R}{{\mathbb R}}

\renewcommand{\(}{\left(}
\renewcommand{\)}{\right)}

\newcommand{\hw}{\hat{w}}

\renewcommand{\th}{\theta}

\newtheorem{theorem}{\bf Theorem}[section]
\newtheorem{lemma}[theorem]{\bf Lemma}
\newtheorem{proposition}[theorem]{\bf Proposition}

\theoremstyle{remark}
  \newtheorem{remark}[theorem]{\sc Remark}
\theoremstyle{definition}
  
  \newtheorem{example}[theorem]{\sc Example}
\numberwithin{equation}{section}

\begin{document}


\title{On the wave-like energy estimates of Klein-Gordon type equations with time dependent potential}
\author{
Kazunori Goto
\footnote{Graduate School of Sciences and Technology for Innovation, 
Yamaguchi University, Japan; e-mail: b003vbv@yamaguchi-u.ac.jp}
and
Fumihiko Hirosawa\footnote{Department of Mathematical Sciences, Faculty of Science, Yamaguchi University, Japan; 
e-mail: hirosawa@yamaguchi-u.ac.jp}
}
\date{}
\maketitle

\begin{abstract}
We consider the conditions for the time dependent potential in which the energy of the Cauchy problem of Klein-Gordon type equation asymptotically behaves like the energy of the wave equation. 
The conclusion of this paper is that the condition is not always given by the order of the potential itself, but should be given by ``generalized zero mean condition'', which is represented by the integral of the potential. 
We also introduce ``generalized modified energy conservation'' in order to describe the appropriate energy for our problem.
\end{abstract}

%
\section{Introduction}
%

Let us consider the following Cauchy problem for Klein-Gordon type equation with time dependent potential:
\begin{equation}\label{KG}
\begin{cases}
\(\pa_t^2 - \Delta + M(t)\) u(t,x) = 0, & (t,x)\in (0,\infty)\times \R^n,
\\ 
u(0,x)=u_0(x),\;\; \pa_t u(0,x)=u_1(x), & x \in \R^n,
\end{cases}
\end{equation}
where $\Delta$ denotes the Laplace operator in $\R^n$ and the potential $M$ is real valued but not necessarily a definite sign. 
It may be natural that $M$ is positive from the point of view of the physical model, but we study it as a mathematical model and remove the restriction. 

It is well known that the energy conservation holds if $M$ is a non-negative constant, and in the case of general $M$, the following property of generalized energy conservation of Klein-Gordon type is proved in \cite{B11, BR12}: 
\begin{equation}\label{GEC}
  q(t)^2 E_{KG}(u;p)(0) \lesssim E_{KG}(u;p)(t) \lesssim E_{KG}(u;p)(0)
\end{equation}
for positive decreasing functions $p$ and $q$ under appropriate conditions to $M$, where 
\begin{equation*}
  E_{KG}(u;p):=
    \|\nabla u(t,\cd)\|_{L^2}^2 + \|\pa_t u(t,\cd)\|_{L^2}^2
   +p(t) \|u(t,\cd)\|_{L^2}^2. 
\end{equation*}
More precisely, if $M=\mu^2(1+t)^{-2\nu}$ with $\mu>0$ and $0 \le \nu \le 1$, then $p$ and $q$ are given by $p=(1+t)^{-\nu_1}$ with $\nu_1<2$ and $q=(1+t)^{-\nu_2}$, respectively, where $\nu_1$ and $\nu_2$ are determined by $\mu$ and $\nu$. 
On the other hand, if $\nu>1$, that is, $\sqrt{|M|} \in L^1([0,\infty))$, then the solution has more wave-like property. 
In \cite{HN18}, the following model is studied as a perturbation problem of \cite{BR12}: 
\begin{equation}\label{M-HN}
  M = \mu^2(1+t)^{-2}+\de(t).
\end{equation}
A conclusion of \cite{HN18} is that the same estimate from above in \eqref{GEC} as in the case $\de = 0$ is valid under some suitable assumptions to $\de(t)$ which permit $\limsup_{t\to\infty} (1+t)^2 \de(t)=\infty$. 
The main purposes of this paper is to determine the conditions for $\de(t)$ of \eqref{M-HN} with $\mu=0$ 
that the \textit{generalized modified energy conservation} of wave type defined later, is established. 
From another point of view, we will determine \textit{generalized zero mean condition} for $M$ 
that \eqref{KG} has wave-like property in spite of $\sqrt{|M|} \not\in L^1([0,\infty))$. 

%
\section{Main theorem}
%
For $b \in C^0([0,\infty))$ satisfying $\lim_{t\to\infty} b(t)=0$ and large $T$, we define the modified energy of the wave type $E(u;b)$ and 
\textit{generalized modified energy conservation} by 
\begin{equation*}
  E(u;b)(t):= 
  \left\|\nabla u(t,\cd) \right\|_{L^2}^2
 +\left\|\pa_t u(t,\cd) - b(t) u(t,\cd)\right\|_{L^2}^2
\end{equation*}
and
\begin{equation}\label{MGEC}
  E(u;b)(t) \simeq E(u;b)(T) \;\; (t \ge T).
\end{equation}

For $M$, we introduce the following properties with parameters $\al$, $\be$ and $\ga$: 
\begin{description}
\item[(M1)]
For $\al \le 1$: 
\begin{equation}\label{stb_0}
  \int^t_0 \left|
    \int^\infty_s \int^\infty_\si M(\tau)\,d\tau \,d\si \right| \,ds
  \lesssim (1+t)^\al
  \;\;(\al \ge 0)
\end{equation}
\begin{equation}\label{stb_inf}
  \int^\infty_t \left|
    \int^\infty_s \int^\infty_\si M(\tau)\,d\tau \,d\si \right| \,ds
  \lesssim (1+t)^\al
  \;\;(\al \le 0). 
\end{equation}
\item[(M2)] 
For $\be < 1$: 
\begin{equation}\label{M0}
  \left|M(t)\right| \lesssim (1+t)^{-2\be}.
\end{equation}
\item[(M3)] 
For $\ga > 0$: 
\begin{equation}\label{intM}
  \left|\int^{\infty}_t M(s)\,ds\right| \lesssim (1+t)^{-\ga},
\end{equation}
\begin{equation}\label{Q2}
  \int^\infty_t \(\int^\infty_s M(\si)\,d\si\)^2 \,ds
  \lesssim (1+t)^{-\ga}
\end{equation}
and
\begin{equation}\label{intQ2}
  \int^\infty_0 \int^\infty_t 
  \(\int^\infty_s M(\si)\,d\si\)^2 \,ds \,dt
  <\infty.
\end{equation}
\end{description}

\begin{remark}\label{ga-ge-be}
\begin{enumerate}
\item[(i)]  
The following estimate is implicitly assumed in \eqref{stb_0}: 
\begin{equation}\label{intQ1}
  \left|\int^\infty_0 \int^\infty_t M(s)\,ds \,dt\right|<\infty.
\end{equation}
\item[(ii)]
If \eqref{Q2} holds for $\ga>1$, then \eqref{stb_0} with $\al=-\ga+2$ is trivial . 
Moreover, if $\ga >2$, then \eqref{stb_inf} with $\al=-\ga+2$ is trivial. 
\end{enumerate}
\end{remark}

\begin{theorem}
Let $u_0 \in H^1$ and $u_1 \in L^2$. 
If {\rm (M1)}, {\rm (M2)} and {\rm (M3)} are valid for 
\begin{equation}\label{al-be}
  \ga \ge 
  \be 
  \begin{cases}
    \ge (\al + 1)/2 &\text{ for } \al \neq 0,\\
    > 1/2 &\text{ for } \al = 0, 
  \end{cases}
\end{equation}
and the following estimate holds:
\begin{equation}\label{eq1-Thm}
  \sup_{t \ge 0}\left\{(1+t)^\al \int^\infty_t \int^\infty_s 
  \(\int^\infty_\si M(\tau)\,d\tau\)^2 \,d\si \,ds\right\} <\infty,
\end{equation}
then there exist $T>0$ and $b \in C^0([0,\infty))$ satisfying 
$b(t) \lesssim (1+t)^{-\ga}$ such that 
\eqref{MGEC} is established. 
Moreover, the following estimate is established for any $t \ge 0$: 
\begin{equation}\label{eq2-Thm}
  E(u;b)(t) \lesssim E_{KG}(u;1)(0). 
\end{equation}
\end{theorem}

In \cite{HN18}, $\be \ge (-\ga+3)/2$ is assumed instead of \eqref{al-be} without assuming (M1), that is, only the trivial case $\al = -\ga +2$ in Remark \ref{ga-ge-be} (ii), is considered. 
The following $M$ is an example of the non-trivial case $\al<-\ga+2$. 

\begin{example}
Let $M(t):=\frac{d}{dt}(\sin((1+t)^\ka)(1+t)^{-2\be-\ka+1})$ 
with $\be\le 1/2$ and $\ka > 2(1-\be)$. 
Noting the estimates 
$|M(t)|\lesssim (1+t)^{-2\be}$, 
$\int^\infty_t M(s)\,ds=\sin((1+t)^\ka)(1+t)^{-2\be-\ka+1}$ 
and 
$|\int^\infty_t \int^\infty_s M(\si)\,d\si\,ds| \lesssim (1+t)^{-2\be-2\ka+2}$,  (M2) and (M3) are valid for 
$\ga=2\be+\ka-1$ and $\al=-2\be-2\ka+3 = -\ga-\ka+2$, 
it follows that $\al < -\ga + 2$. 
Moreover, \eqref{al-be} and \eqref{eq1-Thm} are valid by
$\ga > 1 > \be =(\al+1)/2 + 2\be+\ka-2 > (\al+1)/2$
and 
$\al-2\ga+2 = -6\be-4\ka+7 < 2\be-1 \le 0$. 
\end{example}

The conditions (M1)-(M3) seem to be artificial, but they can be actually natural from the viewpoint of previous studies. 
(M1) and (M2) correspond to \textit{stabilization property} and \textit{$C^2$-property} with \textit{very fast oscillation}, respectively, which were introduced in \cite{EFH15, H07} for the energy estimate of the wave equation with time dependent propagation speed. 
\eqref{Q2} and \eqref{intQ1} are corresponding to \textit{generalized zero mean condition}, which was introduced in \cite{HW08}. 
Moreover, \eqref{intQ2} is considered to be related to the classification of scale invariant potential for the Klein-Gordon type equation in \cite{BR12}.

%
\section{Proof of the theorem}
%
The proof of the theorem is based on the methods introduced in \cite{EKNR14,HN18,H19} that the Klein-Gordon type equation is reduced to a dissipative wave equation or a wave equation with time dependent propagation speed. 
Then, solutions of the equations are estimated in a particular zones of time-frequency space by the method introduced in \cite{H07, HW08} after the Fourier transformation with respect to spatial variables. 

\subsection{Reduction to a dissipative wave equation}
For $t \ge T$ with a large $T$, we reduce the Klein-Gordon type equation of \eqref{KG} to the dissipative wave equation 
$(\pa_t^2 - \Delta + 2b(t)\pa_t ) w = 0$ 
by the transformation 
\begin{equation*}
  w(t,x):=\exp\(\int^\infty_t b(s)\,ds\)u(t,x),
\end{equation*}
where $b$ is a solution of the following Riccati equation
\begin{equation}\label{b-M}
  b'(t)+b(t)^2+M(t)=0. 
\end{equation}

Let us derive the representation of a particular solution of \eqref{b-M}. 
We define $\{q_k(t)\}_{k=1}^\infty$ and $\{Q_k(t)\}_{k=1}^\infty$ on $[T,\infty)$ by 
\begin{equation*}\label{qk}
  q_1(t):=M(t),\;\;
  q_k(t):=\sum_{j=1}^{k-1} Q_j(t)Q_{k-j}(t) \;\; (k=2,3,\ldots)
\end{equation*}
and
\begin{equation*}\label{Qk}
  Q_k(t):=-\int^\infty_t q_k(s)\,ds
  \;\;(k=1,2,\ldots).
\end{equation*}
%
\begin{lemma}\label{sol_b-M}
A particular solution of \eqref{b-M} is represented by 
$b(t):=\sum_{k=1}^\infty Q_k(t)$. 
\end{lemma}
\begin{proof}
The proof is straightforward calculation. 
\end{proof}

The following lemmas ensure the convergence of $b(t)$ on $[T,\infty)$ for large $T$. 
\begin{lemma}\label{Lem-QkQ2}
$Q_2(t) \le 0$ and the following estimate is established for any $k\ge 2:$ 
\begin{equation}\label{Lem-eq-QkQ2}
  |Q_k(t)| \le 4^{k-1}(-Q_2(t)) \phi(t)^{\frac{k-2}{2}},
  \;\;
  \phi(t):=-\int^\infty_t Q_2(s)\,ds.
\end{equation}
\end{lemma}
\begin{proof}
$Q_2(t) \le 0$ is trvial from the definition, 
and $\lim_{t\to\infty}Q_2(t)=0$ by \eqref{Q2}. 
\eqref{Lem-eq-QkQ2} is trivial for $k=2$. 
If \eqref{Lem-eq-QkQ2} is valid for $k=3,\ldots,l$, 
then by Cauchy-Schwarz inequality, integration by parts, 
noting $\phi'(t)=Q_2(t)$ and $\frac{d}{dt} Q_2(t)=q_2(t) \ge 0$, we have
\begin{align*}
  \left|\int_t^\infty Q_1(s) Q_l(s) \,ds \right|
\le & \ 4^{l-1} \(\int_t^\infty Q_1(s)^2 \,ds \)^{\frac12}
  \(\int_t^\infty
         Q_2(s)^2 \phi(s)^{l-2} \,ds \)^{\frac12}
\\
=& \ \frac{4^{l-1}}{(l-1)^{\frac12}} (-Q_2(t))^{\frac12}
  \(\int_t^\infty Q_2(s) \frac{d}{ds}\phi(s)^{l-1}
   \,ds \)^{\frac12}
\\
\le & \ 4^{l-1} (-Q_2(t)) \phi(t)^{\frac{l-1}{2}}. 
\end{align*}
Moreover, for $2 \le j \le l$ we have 
\begin{align*}
  \left| \int_t^\infty Q_j(s)Q_{l+1-j}(s) \,ds \right|
  \le & \
   4^{l-1} 
   \int_t^\infty Q_2(s)^2 \phi(s)^{\frac{l-3}{2}} \,ds 
\\
  = & \
   \frac{4^{l}}{2(l-1)}
   \int_t^\infty Q_2(s)\frac{d}{ds}\phi(s)^{\frac{l-1}{2}} \,ds 
  \le
  \frac{4^{l}(-Q_2(t))\phi(t)^{\frac{l-1}{2}} }{2(l-1)} .
\end{align*}
Therefore, we obtain
\begin{align*}
  |Q_{l+1}(t)|
  \le & \
    2\left| \int_t^\infty Q_1(\si)Q_{l}(\si) \,d\si\right|
 +\sum_{j=2}^{l-1}
  \left| \int_t^\infty Q_j(\si)Q_{l+1-j}(\si) \,d\si\right|
\\
  = & \
  4^{l}(-Q_2(t))\phi(t)^{\frac{l-1}{2}}
  \(
    \frac{1}{2} + \sum_{j=2}^{l-1} \frac{1}{2(l-1)} 
  \)
  \le
  4^l (-Q_2(t))\phi(t)^{\frac{l-1}{2}},
\end{align*}
it follows that \eqref{Lem-eq-QkQ2} is also valid for any $k \le l+1$. 
Thus \eqref{Lem-eq-QkQ2} is valid for any $k \ge 2$. 
\end{proof}

\begin{lemma}\label{Lemma_estb}
There exist positive constants $T$, $b_0=b_0(T)$, $b_1=b_1(T)$ and $b_2=b_2(T)$ such that the following estimates are established for any $t \ge T$: 
\begin{equation}\label{est_sumQk}
  \sum_{k=2}^\infty \left|Q_k(t)\right| \le \frac32 |Q_2(t)|,
\end{equation}
\begin{equation}\label{int_b}
  \left|\int_t^\infty b(s)\,ds\right| \le b_0,
\end{equation}
\begin{equation}\label{est|b|}
  |b(t)| \le b_1 (1+t)^{-\ga}
\end{equation}
and
\begin{equation}\label{estb'}
  |b'(t)| \le b_2 (1+t)^{-2\be}.
\end{equation}
\end{lemma}
\begin{proof}
By \eqref{intM}, \eqref{intQ2} and \eqref{intQ1} there exists $T>0$ such that 
\begin{equation}\label{set_T}
  \left|\int^\infty _t Q_1(s)\,ds \right| \le 1
  \;\text{ and }\;
  \phi(t) \le \frac{1}{6^4}
\end{equation}
for any $t \ge T$. 
Then, by Lemma \ref{Lem-QkQ2}, we have 
\begin{align*}
  \left|\sum_{k=3}^\infty Q_k(t) \right|
  \le 
  |Q_2(t)| \sum_{k=3}^\infty 
  4^{k-1} \phi(t)^{\frac{k-2}{2}}
  \le
  \frac12 |Q_2(t)|, 
\end{align*}
which gives \eqref{est_sumQk}. 
By \eqref{intQ2}, \eqref{intQ1} and \eqref{est_sumQk}, we have 
\begin{equation}\label{est_int|b|}
  \left|\int_t^\infty b(s)\,ds\right|
  \le 
  \left|\int^\infty_t Q_1(s)\,ds\right|
 +\sum_{k=2}^\infty  \int^\infty_t \left| Q_k(s)\right|\,ds
  \le b_0. 
\end{equation}
By \eqref{intM}, \eqref{Q2} and \eqref{est_sumQk}, we have 
\begin{align*}
  |b(t)|
  \le
  |Q_1(t)| + \frac32|Q_2(t)| \le b_1(1+t)^{-\ga}. 
\end{align*}
By \eqref{M0}, \eqref{al-be}, \eqref{b-M} and \eqref{est|b|}, we have 
\begin{equation*}
  |b'(t)| 
  \le 
  |M(t)| + b_1^2(1+t)^{-2\ga}
  \le b_2(1+t)^{-2\be}. 
\end{equation*}
Thus the proof is concluded. 
\end{proof}

Lemma \ref{Lemma_estb} ensures that the solution of \eqref{KG} is represented by the solution of the following dissipative wave equation: 
\begin{equation}\label{DW}
  (\pa_t^2 - \Delta + 2b(t)\pa_t) w(t,x) = 0
\end{equation}
for $t \ge T$. 
By carrying out partial Fourier transformation with respect to spatial variables and denoting the Fourier image of $w(t,x)$ as $\hw(t,\xi)$, 
\eqref{DW} is represented as follows: 
\begin{equation}\label{FDW}
(\pa_t^2 + |\xi|^2 + 2b(t) \pa_t) \hw(t,\xi) = 0. 
\end{equation}
Moreover, \eqref{FDW} is represented by the following first order system: 
\begin{equation}\label{W}
  \pa_t W = A W, 
  \;\;
  A:= \begin{pmatrix} -2b(t) & i|\xi| \\ i|\xi| & 0 \end{pmatrix},
  \;\;
  W:=\begin{pmatrix} \pa_t w \\ i|\xi| w \end{pmatrix}.
\end{equation}
We estimate the solution of \eqref{W} in different ways in the following two zones of the time-frequency space $[T,\infty)\times \R^n$: 
\begin{equation*}
\begin{cases}
  Z_H:=\left\{(t,\xi) \in [T,\infty)\times \R^n\;;\; 
  (1+t)^\al |\xi| \ge N \right\}, \\
  Z_\Psi:=\left\{(t,\xi)\in [T,\infty)\times \R^n \;;\; 
  (1+t)^\al |\xi| \le N \right\},
\end{cases}  
\end{equation*}
where $N$ is a positive constant which will be chosen later. 
Denoting 
\[
  t_\xi:=\max\left\{T,(N|\xi|^{-1})^{\frac{1}{\al}}-1\right\}
\]
for $\al\neq 0$ and $|\xi|>0$, 
we see that 
$Z_H=\{ t \ge t_\xi\}$ and $Z_\Psi=\{ T \le t \le t_\xi\}$ for 
$\al > 0$, and that 
$Z_H=\{ T \le t \le t_\xi\}$ and $Z_\Psi=\{ t \ge t_\xi\}$ for $\al < 0$. 

\subsection{Estimate in $Z_H$}

\begin{proposition}\label{prop_estZH}
There exist positive constants $N$ and $K_1$ such that the following estimates are established in $Z_H$: 
\begin{equation*}
  \begin{cases}
    K_1^{-1} |W(t_\xi,\xi)| \le |W(t,\xi)| \le K_1|W(t_\xi,\xi)| 
      & (\al>0), \\
    K_1^{-1} |W(T,\xi)|\le |W(t,\xi)| \le K_1|W(T,\xi)| 
      & (\al \le 0).
  \end{cases}
\end{equation*}
\end{proposition}
\begin{proof}
Let $(t,\xi) \in Z_H$. 
Setting $N \ge 2b_1$, by \eqref{al-be} and \eqref{est|b|} we have 
$|b(t)| \le b_1(1+t)^{-\ga} \le b_1 (1+t)^{-\al} 
  \le b_1 N^{-1}|\xi| \le |\xi|/2$. 
Since the eigenvalues and the respective eigenvectors of $A$ are given by 
$\{\la,\ol{\la}\}$ and 
$\{{}^t(1,i\de),{}^t(-i\de,1)\}$, where 
$\la = -b(t) - i\sqrt{|\xi|^2-b^2}$ and $\de = \la|\xi|^{-1}$,
and noting the inequalities 
\begin{equation}\label{est_detM}
  \sqrt{3} \le |1-\de^2| = 2\sqrt{1-b^2|\xi|^{-2}} \le 2,
\end{equation}
$A$ is diagonalized as $M^{-1}A M = \diag(\la,\ol{\la})=:\La$ by the diagonalizer 
\begin{equation*}
  M:=\begin{pmatrix} 1 & -i\de \\ i \de & 1 \end{pmatrix}. 
\end{equation*}
Denoting $W_1:=M^{-1}W$, \eqref{W} is rewritten as follows: 
\begin{equation*}
  \pa_t W_1 
  = \(\La - M^{-1} (\pa_t M) \)W_1
  = \(\La_1 + R_1\)W_1, 
\end{equation*}
where
\begin{equation*}
  \La_1 
  :=\(-b - \frac{\pa_t \log\(1-\de^2\)}{2} \)I 
   - i \sqrt{|\xi|^2-b^2} \begin{pmatrix} 1 & 0 \\ 0 & -1 \end{pmatrix}
\end{equation*}
and
\begin{equation*}
  R_1:= -i\frac{b'|\xi|^{-1}}{2\(1-b^2|\xi|^{-2}\)} 
  \begin{pmatrix} 0 & -1 \\ 1 & 0 \end{pmatrix}.
\end{equation*}
Then, by \eqref{est_detM} we have 
\begin{align*}
  \pa_t |W_1|^2
  \lesseqgtr
  -\(2b+\Re\(\pa_t \log\(1-\de^2\)\) \pm \frac{4}{3}|b'||\xi|^{-1} \) |W_1|^2. 
\end{align*}
Noting that \eqref{al-be} and \eqref{estb'} conclude the following estimates: 
\begin{equation*}
|\xi|^{-1} \int^\infty_{t_\xi} |b'(s)|\,ds<\infty \;(\al \ge 0)
\;\text{ and }\;
|\xi|^{-1} \int^{t_\xi}_{T} |b'(s)|\,ds<\infty \; (\al < 0),
\end{equation*}
by Lemma \ref{Lemma_estb}, \eqref{est_detM} and Gronwall's inequality, we have 
$|W_1(t,\xi)| \simeq |W_1(t_\xi,\xi)|$ for $\al>0$ 
and 
$|W_1(t,\xi)| \simeq |W_1(T,\xi)|$ for $\al\le 0$ in $Z_H$. 
Finally, noting that $|\de|^2 = 1$ and \eqref{est_detM} gives 
$\sqrt{1/2}|W| \le |W_1| \le \sqrt{2/3}|W|$, 
we conclude the proof. 
\end{proof}

\subsection{Estimate in $Z_\Psi$}

\begin{proposition}\label{prop_estZPsi}
There exist a positive constant $K_2$ such that the following estimates are established in $Z_\Psi$: 
\begin{equation*}
  \begin{cases}
    K_2^{-1}|W(T,\xi)| \le |W(t,\xi)| \le K_2|W(T,\xi)| 
      & (\al>0), \\
    K_2^{-1}|W(t_\xi,\xi)| \le |W(t,\xi)| \le K_2|W(t_\xi,\xi)| 
      & (\al \le 0).
  \end{cases}
\end{equation*}
\end{proposition}
\begin{proof}
Let us introduce the change of variable from $t \in [T,\infty)$ to $\th \in [0,\infty)$ by 
\begin{equation*}
  \th: = \int^t_T \exp\(-2\int^s_T b(\si)\,d\si\)\,ds.
\end{equation*}
Here we note that $\th(t)$ is strictly increasing and satisfying 
$e^{-2b_0} t \le \th(t) \le e^{2b_0} t$ by \eqref{int_b}. 
We define $a(\tau)$ and $\eta(\tau)$ by 
\[
  a(\tau):=\exp\(2\int^{\th^{-1}(\tau)}_T b(s)\,ds\)
\]
and
\[
  \eta(\tau):=\exp\(
    2\int^\infty_T Q_1(s)\,ds 
   +2\int^{\th^{-1}(\tau)}_T \sum_{k=2}^\infty Q_k(s)\,ds\). 
\]
Here we remark that 
$e^{-2b_0} \le a(\tau)$, 
$\eta(\tau) \le e^{2b_0}$ and $\eta'(\tau) \ge 0$ 
are valid by Lemma \ref{Lem-QkQ2}, \eqref{est_sumQk} and \eqref{est_int|b|}. 
By mean value theorem, \eqref{stb_0} and \eqref{set_T}, 
there exist constants $a_0>0$ and $0<\ka<1$ such that the following estimates are established: 
\begin{align*}
\int^{\th(t)}_{\th(T)} |a(\si)-\eta(\si)|\,d\si
=& \
  \int^{t}_T \left|1 - \exp\(2\int^\infty_s Q_1(\si)\,d\si\)\right|\,ds
\\
=& \ \int^{t}_T 
  \left| 2\int^\infty_s Q_1(\si)\,d\si
  \exp\(2 \ka \int^\infty_s Q_1(\si)\,d\si\)\right|\,ds
\\
\le & \ 
  2 e^{2 \ka} \int^{t}_0 
  \left| \int^\infty_s Q_1(\si)\,d\si \right|\,ds
\le a_0 (1+t)^\al.
\end{align*}
Moreover, if \eqref{stb_inf} holds for $\al \le 0$, then we have 
$\int_{\th(t)}^\infty |a(\si)-\eta(\si)|\,d\si \le a_0 (1+t)^\al$. 

By the change of variables $t \to \th$ and denoting 
$y(\th(t),\xi)=\hw(t,\xi)$, \eqref{W} is represented by 
\[
  \pa_\th  Y = B Y,\;\;
  \;\;
  Y:=\begin{pmatrix} 
    \pa_\th y + i \eta |\xi| y \\ \pa_\th y - i \eta |\xi| y
  \end{pmatrix}
\]
and
\[
  B: = 
  \frac{i|\xi|\(a^2+\eta^2\)}{2\eta}
  \begin{pmatrix} 1 & 0 \\ 0 & -1 \end{pmatrix}
+\frac{\eta'}{2\eta} \begin{pmatrix} 1 & -1 \\ -1 & 1 \end{pmatrix}
+\frac{i|\xi|\(a^2-\eta^2\)}{2\eta} 
  \begin{pmatrix} 0 & -1 \\ 1 & 0\end{pmatrix}. 
\]
Then, by Lemma \ref{Lemma_estb}, we have 
\begin{align*}
  \pa_\th|Y|^2
=& \frac{\eta'}{\eta}\(|Y_1|^2 - 2\Re\(\ol{Y_1}Y_2\) +|Y_2|^2\)
  +\frac{|\xi|\(a+\eta\)\(a-\eta\)}{\eta} 
   2\Re\(iY_1 \ol{Y_2}\)
\\
\lesseqgtr & \ \(\frac{2\eta'}{\eta} \pm 2e^{4b_0}|\xi||a-\eta|\) |Y|^2.
\end{align*}
Therefore, noting the following estimates: 
\begin{equation*}
\begin{cases}
  |\xi|\int^{\th(t)}_{\th(T)}|a(s)-\eta(s)|\,ds 
    \le a_0|\xi|(1+t)^\al \le a_0 N & (\al \ge 0),
  \\[1mm] 
  |\xi|\int^{\th(t)}_{\th(t_\xi)}|a(s)-\eta(s)|\,ds 
    \le a_0|\xi|(1+t_\xi)^\al = a_0 N & (\al < 0),
\end{cases}
\end{equation*}
by Gronwall's lemma, we have
\begin{align*}
  |Y(\th(t),\xi)|^2 
  \lesseqgtr 
  \(\frac{\eta(\th(t))}{\eta(\th(T))}\)^2
  \exp\(\pm 2a_0 N e^{4b_0}\) |Y(\th(T),\xi)|^2
  \simeq |Y(\th(T),\xi)|^2
\end{align*}
for $\al>0$ and $T \le t \le t_\xi$, and
\begin{align*}
  |Y(\th(t),\xi)|^2 
  \lesseqgtr
  \(\frac{\eta(\th(t))}{\eta(\th(t_\xi))}\)^2
  \exp\(\pm 2a_0 N e^{4b_0}\) |Y(\th(t_\xi),\xi)|^2
  \simeq |Y(\th(t_\xi),\xi)|^2
\end{align*}
for $\al \le 0$ and $t \ge t_\xi$. 
Noting the equalities 
$|Y|^2=2|\pa_\th y|^2+\eta^2|\xi|^2|y|$ and 
$a \pa_\th y = \pa_t \hw$, 
we have $2e^{-4b_0}|W| \le |Y|^2 \le 2e^{4b_0}|W|$, 
and thus we conclude the proof. 
\end{proof}

\subsection{Completion of the proof}

If $\al>0$, then by Proposition \ref{prop_estZH} and Proposition \ref{prop_estZPsi}, we have 
\begin{equation*}
\begin{cases}
  K_2^{-1} |W(T,\xi)| \le |W(t,\xi)| \le K_2 |W(T,\xi)| 
  & (T \le t \le t_\xi),
\\
  |W(t,\xi)|
  \begin{cases}
    \le K_1 |W(t_\xi,\xi)| \le K_1 K_2 |W(T,\xi)| \\
    \ge K_1^{-1}|W(t_\xi,\xi)| \ge K_1^{-1} K_2^{-1} |W(T,\xi)|
  \end{cases}
  & (t \ge t_\xi). 
\end{cases}
\end{equation*}
On the other hand, if $\al \le 0$, then we have
\begin{equation*}
\begin{cases}
  K_1^{-1} |W(T,\xi)| \le |W(t,\xi)| \le K_1 |W(T,\xi)| 
  & (T \le t \le t_\xi),
\\
  |W(t,\xi)| \le 
  \begin{cases}
    \le K_2 |W(t_\xi,\xi)| \le K_1 K_2 |W(T,\xi)| \\
    \ge K_2^{-1}|W(t_\xi,\xi)| \ge K_2^{-1} K_2^{-1} |W(T,\xi)| 
  \end{cases}
  &(t \ge t_\xi). 
\end{cases}
\end{equation*}
Consequently, since the estimate 
$E(u;b)(t) = \exp\(-2\int^\infty_t b(s)\,ds\)\|W(t,\cd)\|^2_{L^2}
  \simeq \|W(t,\cd)\|^2_{L^2}$
holds by \eqref{int_b} and Parseval's equality, we have \eqref{MGEC}.

In order to prove \eqref{eq2-Thm}, introduce the following proposition: 
\begin{proposition}\label{prop_est0T}
For any $T>0$, there exists a positive constant $K_0=K_0(T)$ such that the following estimate is established on $[0,T]$: 
\begin{equation*}
  E_{KG}(u;1)(t)
  \le K_0
  E_{KG}(u;1)(0).
\end{equation*}
\end{proposition}
\begin{proof}
We extend $b(t)$ on $[0,T)$ as $b \in C^0([0,\infty))$ and $|b|$ is monotone decreasing. 
By Cauchy-Schwarz inequality, we have 
\begin{align*}
  \frac{d}{dt}E_{KG}(u;1)(t)
= (1-M(t))\Re\(u(t,\cd),\pa_t u(t,\cd)\)_{L^2}
\le |1-M(t)| E_{KG}(u;1)(t).
\end{align*}
Therefore, by \eqref{M0} and Gronwall's inequality, we have 
\begin{align*}
  E(u;b)(t)
\le& \ 
 \|u(t,\cd)\|_{L^2}^2 +2\|\pa_t u(t,\cd)\|^2 
    +2b(T)^2\|u(t,\cd)\|_{L^2}^2
\\
\simeq & \ E_{KG}(u;1)(t)
  \le \exp\(T \sup_{0\le t \le T}\{|1-M(s)|\}\) E_{KG}(u;1)(0). 
\end{align*}
for any $t \in [0,T]$. 
\end{proof}
Thus \eqref{eq2-Thm} is proved by combining Proposition \ref{prop_est0T} and \eqref{MGEC}. 



\end{document}